\documentclass[MSNbibl,number,citesort,dvips]{arxbj}
\usepackage{mathrsfs}

\aid{0}
\volume{20}
\issue{2}
\pubyear{2014}
\firstpage{586}
\lastpage{603}
\doi{10.3150/12-BEJ498} 

\makeatletter

\newcommand{\R}{\mathbb{R}}
\newcommand{\N}{\mathbb{N}}

\newcommand{\ca}{\mathcal{A}}
\newcommand{\cd}{\mathcal{D}}

\newcommand{\ga}{\gamma}
\newcommand{\si}{\sigma}
\newcommand{\ff}{\varphi}
\newcommand{\poly}{Q}

\newtheorem{thmm}{Theorem}[section]
\newtheorem{theorem}{Theorem}[section]
\newproclaim{definition}[theorem]{Definition}
\newtheorem{lemma}[theorem]{Lemma}
\newtheorem{proposition}[theorem]{Proposition}

\makeatother

\begin{document}
\begin{frontmatter}

\title{Invariance principles for homogeneous sums of free random variables}
\runtitle{Invariance principles for homogeneous sums of free random variables}

\begin{aug}
\author{\fnms{Aur\'elien}~\snm{Deya}\corref{}\thanksref{e1}\ead[label=e1,mark]{aurelien.deya@iecn.u-nancy.fr}}
\and
\author{\fnms{Ivan}~\snm{Nourdin}\thanksref{e2}\ead[label=e2,mark]{inourdin@gmail.com}}
\runauthor{A. Deya and I. Nourdin} 
\address{Institut \'Elie Cartan, Universit\' e
de Lorraine, Campus Aiguillettes, BP 70239, 54506 Vandoeuvre-l\`
es-Nancy, France. \printead{e1,e2}}

\end{aug}

\received{\smonth{1} \syear{2012}}
\revised{\smonth{10} \syear{2012}}

%
\begin{abstract}
We extend, in the free probability framework,
an invariance principle for multilinear homogeneous sums with low influences
recently established by Mossel, O'Donnel and Oleszkiewicz in
[\textit{Ann. of Math.} (2) \textbf{171} (2010) 295--341].
We then deduce
several universality phenomenons, in the spirit of the paper
[\textit{Ann. Probab.} \textbf{38} (2010) 1947--1985] by
Nourdin, Peccati and Reinert.
\end{abstract}

%

\end{frontmatter}

\section{Introduction and background}

\textit{Motivation and main goal}.
Our starting point is
the following weak version (which is enough for our purpose) of an
invariance principle for multilinear homogeneous sums with low influences,
recently established in \cite{MOO}.
%
%
\begin{thmm}[(Mossel--O'Donnel--Oleszkiewicz)]\label{MOO}
Let $(\Omega,\mathcal{F},P)$ be a probability space (in the classical sense).
Let $X_1,X_2,\ldots$ (resp., $Y_1,Y_2,\ldots$) be a sequence of
independent centered random variables with unit variance
satisfying
moreover
\[
\sup_{i\geq1}E \bigl[|X_i|^r \bigr]<\infty \qquad\Bigl(\mbox{resp., }\sup_{i\geq
1}E \bigl[|Y_i|^r \bigr]<
\infty \Bigr)\mbox{ for all $r\geq1$}.
\]
Fix $d\geq1$, and consider a sequence of functions $f_N\dvt\{1,\ldots
,N\}^d\to\R$
satisfying the following two assumptions for each $N$ and each
$i_1,\ldots,i_d=1,\ldots,N$:
\begin{enumerate}[(ii)]
\item[(i)] {(full symmetry)} $ f_N(i_1,\ldots,i_d)=f_N(i_{\sigma
(1)},\ldots,i_{\sigma(d)})$
for all $\sigma\in\mathfrak{S}_d$;
%
\item[(ii)] {(normalization)} $ d!\sum_{j_1,\ldots,j_d=1}^N
f_N(j_1,\ldots,j_d)^2=1$.
\end{enumerate}
Also, set
%
%
\begin{equation}
\label{qn} Q_N(x_1,\ldots,x_N)=\sum
_{i_1,\ldots,i_d=1}^N f_N(i_1,
\ldots,i_d) x_{i_1}\cdots x_{i_d}
\end{equation}
and
\[
\operatorname{Inf}_i(f_N)=\sum_{j_2,\ldots,j_{d}=1}^N
f_N(i,j_2,\ldots,j_{d})^2,\qquad i=1,
\ldots,N.
\]
Then, for any integer $m\geq1$,
%
%
\begin{equation}
\label{MOOfree-eq-1} E \bigl[Q_N(X_1,
\ldots,X_N)^m \bigr] - E \bigl[Q_N(Y_1,
\ldots,Y_N)^m \bigr] = \mathrm{O} \bigl(\tau_N^{1/2}
\bigr),
\end{equation}
where $\tau_N=\max_{1\leq i\leq N} \operatorname{Inf}_i(f_N)$.
\end{thmm}

In \cite{MOO}, the authors were motivated by solving two conjectures,
namely the \textit{Majority Is Stablest} conjecture from theoretical
computer science and the
\textit{It Ain't Over Till It's Over} conjecture from social choice theory.
It is worthwhile noting that there is another striking consequence of
Theorem \ref{MOO}, more in the spirit of the classical
central limit theorem. Indeed, in article \cite{NPR} Nourdin, Peccati
and Reinert combined
Theorem \ref{MOO} with the celebrated \textit{Fourth Moment theorem} of
Nualart and Peccati \cite{NP}, and deduced that
multilinear homogenous sums of general centered
independent random variables with unit variance enjoy the following
universality phenomenon.
%
%
\begin{thmm}[(Nourdin--Peccati--Reinert)]\label{NPR}
Let $(\Omega,\mathcal{F},P)$ be a probability space (in the classical sense).
Let $G_1,G_2,\ldots$ be a sequence of i.i.d. $\mathcal{N}(0,1)$ random
variables.
Fix $d\geq2$ and consider a sequence of functions $f_N\dvt\{1,\ldots
,N\}^d\to\R$
satisfying the following three assumptions for each $N$ and each
$i_1,\ldots,i_d=1,\ldots,N$:
\begin{enumerate}[(iii)]
\item[(i)] {(full symmetry)} $ f_N(i_1,\ldots,i_d)=f_N(i_{\sigma
(1)},\ldots,i_{\sigma(d)})$
for all $\sigma\in\mathfrak{S}_d$;
\item[(ii)] {(vanishing on diagonals)} $ f_N(i_1,\ldots,i_d)=0$ if
$i_k=i_l$ for some $k\neq l$;
\item[(iii)] {(normalization)} $ d!\sum_{j_1,\ldots,j_d=1}^N
f_N(j_1,\ldots,j_d)^2=1$.
\end{enumerate}
Also, let $Q_N(x_1,\ldots,x_N)$ be given by
(\ref{qn}).
Then, the following two conclusions are equivalent as $N\to\infty$:
\begin{enumerate}[(A)]
\item[(A)] $Q_N(G_1,\ldots,G_N) \stackrel{\mathrm{law}}{\to} \mathcal{N}(0,1)$;
\item[(B)] $Q_N(X_1,\ldots,X_N)\stackrel {\mathrm{law}}{\to} \mathcal
{N}(0,1)$ for \textup{any} sequence $X_1,X_2,\ldots$ of i.i.d. centered
random variables with unit variance and all moments.
\end{enumerate}
\end{thmm}

In the present paper, our goal is twofold. We shall first extend
Theorem \ref{MOO} in the context of
free probability and we shall then investigate whether a result such as
Theorem \ref{NPR} continues to hold true in this
framework.
We are motivated by the fact that there is often a close
correspondence between classical probability and free probability, in
which the
Gaussian law (resp., the classical notion of independence) has the
semicircular law (resp., the notion of free independence) as an
analogue.

\textit{Free probability in a nutshell}.
Before going into the details and for the sake of clarity, let us first
introduce some of the central concepts in the theory of free probability.
(See \cite{nicaspeicher} for a systematic presentation.)

A \emph{non-commutative probability space} is a von Neumann algebra
$\mathcal{A}$ (i.e.,
an algebra of operators on a real
separable Hilbert space,
closed under adjoint and convergence in the weak operator topology)
equipped with a \emph{trace} $\ff$,
that is, a unital linear functional (meaning preserving the identity)
which is weakly
continuous, positive (meaning $\ff(X)\ge0$
whenever $X$ is a non-negative element of $\mathcal{A}$; i.e.,
whenever $X=YY^\ast$
for some $Y\in\mathcal{A}$), faithful (meaning that if
$\ff(YY^\ast)=0$ then $Y=0$), and tracial (meaning that $\ff(XY)=\ff
(YX)$ for all
$X,Y\in\mathcal{A}$, even though in general $XY\ne YX$).

In a non-commutative probability space, we refer to the self-adjoint
elements of the
algebra as
\emph{random variables}. Any
random variable $X$ has
a \emph{law}: this is the unique
probability measure $\mu$ on $\R$
with the same moments as $X$; in other words, $\mu$ is such that
%
%
\begin{equation}
\label{mu} \int_{\R} \poly(x) \,\mathrm{d}\mu(x) = \ff \bigl(\poly(X)
\bigr)
\end{equation}
for any real polynomial $\poly$.

In a non-commutative probability setting, the central notion of \emph
{free independence} (introduced by
Voiculescu in \cite{Voiculescu})
goes as follows. Let $\mathcal{A}_1,\ldots,\mathcal{A}_p$ be unital
subalgebras of $\mathcal{A}$. Let $X_1,\ldots, X_m$ be elements
chosen among the $\mathcal{A}_i$'s such that, for $1\le j<m$,
two consecutive elements $X_j$ and $X_{j+1}$ do not come from the same
$\mathcal{A}_i$, and
such that $\ff(X_j)=0$ for each $j$. The subalgebras $\mathcal
{A}_1,\ldots,\mathcal{A}_p$ are said to be \emph{free} or \emph{freely
independent} if, in this circumstance,
%
%
\begin{equation}
\label{free-def} \ff(X_1X_2\cdots X_m) = 0.
\end{equation}
Random variables are called freely independent if the unital algebras
they generate are freely independent. If $X,Y$ are freely independent,
then their joint moments
are determined by the moments of $X$ and $Y$ separately
as in the classical case.

The \emph{semicircular distribution}
$\mathcal{S}(m,\sigma^2)$
with mean $m\in\R$ and variance $\sigma^2>0$
is the probability distribution
\[
\mathcal{S} \bigl(m,\sigma^2 \bigr) (\mathrm{d}x) = \frac{1}{2\pi\sigma^2} \sqrt{4
\sigma^2-(x-m)^2} \mathbf{1}_{\{|x-m|\le2\sigma\}} \,\mathrm{d}x.
\]
If $m=0$, this distribution is symmetric around $0$,
and therefore its odd moments are all $0$. A~simple calculation shows
that the even centered moments are given by
(scaled) Catalan numbers: for non-negative integers $k$,
\[
\int_{m-2\sigma}^{m+2\sigma} (x-m)^{2k} \mathcal{S}
\bigl(m,\sigma^2 \bigr) (\mathrm{d}x) = C_k \sigma^{2k},
\]
where $C_k = \frac{1}{k+1} \bigl({{2k}\atop{k}} \bigr)$
(see, e.g., \cite{nicaspeicher}, Lecture 2).

\textit{Our main results}.
We are now in a position to state our first main result, which is
nothing but a suitable generalization
of Theorem \ref{MOO} in the free probability setting.
%
%
\begin{thmm}\label{MOOfree}
Let $(\mathcal{A},\varphi)$ be a non-commutative probability space.
Let $X_1,X_2,\ldots$ (resp., $Y_1,Y_2,\ldots$) be a sequence of centered
free random variables with unit variance
(i.e., such that $\ff(X_i^2)=\ff(Y_i^2)=1$ for all $i$), satisfying
moreover
\[
\sup_{i\geq1}\ff \bigl(|X_i|^r \bigr)<\infty
\qquad \Bigl(\mbox{resp., }\sup_{i\geq1}\ff \bigl(|Y_i|^r
\bigr)< \infty \Bigr) \mbox{ for all $r\geq1$},
\]
where $ |X|= \sqrt{X^* X}$.
Fix $d\geq1$, and consider a sequence of functions $f_N\dvt\{1,\ldots
,N\}^d\to\R$
satisfying the following three assumptions for each $N$ and each
$i_1,\ldots,i_d=1,\ldots,N$:
\begin{enumerate}[(iii)]
\item[(i)] {(mirror-symmetry)} $ f_N(i_1,\ldots
,i_d)=f_N(i_d,\ldots,i_1)$;
\item[(ii)] {(vanishing on diagonals)} $ f_N(i_1,\ldots,i_d)=0$ if
$i_k=i_l$ for some $k\neq l$;
\item[(iii)] {(normalization)} $ \sum_{j_1,\ldots,j_d=1}^N
f_N(j_1,\ldots,j_d)^2=1$.
\end{enumerate}
Also, set
%
%
\begin{equation}
\label{qnfree} Q_N(x_1,\ldots,x_N)=\sum
_{i_1,\ldots,i_d=1}^N f_N(i_1,
\ldots,i_d) x_{i_1}\cdots x_{i_d}
\end{equation}
and
\[
\operatorname{Inf}_i(f_N)=\sum_{l=1}^d
\sum_{j_1,\ldots,j_{d-1}=1}^N f_N(j_1,
\ldots,j_{l-1},i,j_l,\ldots,j_{d-1})^2,\qquad
i=1,\ldots,N.
\]
Then, for any integer $m\geq1$,
%
%
\begin{equation}
\label{MOOfree-eq} \ff \bigl(Q_N(X_1,
\ldots,X_N)^m \bigr) - \ff \bigl(Q_N(Y_1,
\ldots,Y_N)^m \bigr) = \mathrm{O} \bigl(\tau_N^{1/2}
\bigr),
\end{equation}
where $\tau_N=\max_{1\leq i\leq N} \operatorname{Inf}_i(f_N)$.
\end{thmm}

Due to the lack of commutativity of the variables involved, the proof
of Theorem \ref{MOOfree}
raises new difficulties with respect to
its commutative
counterpart. Moreover, it is worthwhile noting that it contains the
free central limit theorem as an immediate corollary.
Indeed, let us choose $d=1$ (in this case, assumptions (i) and (ii)
are of course immaterial),\vspace*{-3pt} $Y_1,Y_2,\ldots\sim\mathcal{S}(0,1)$ and
$f_N(i)=\frac{1}{\sqrt{N}}$, $i=1,\ldots,N$. We then have
$Q_N(Y_1,\ldots,Y_N)\sim\mathcal{S}(0,1)\stackrel{\mathrm{law}}{=}Y_1$
(thanks to (iii) as well as the fact
that a sum of freely independent semicircular random variables remains
semicircular) and $\tau_N\to0$ as $N\to\infty$,
so that, thanks to (\ref{MOOfree-eq}),
\[
\ff \biggl[ \biggl(\frac{X_1+\cdots+X_N}{\sqrt{N}} \biggr)^m \biggr]\to\ff
\bigl(Y_1^m \bigr)
\]
for each $m\geq1$ as $N\to\infty$, which is exactly what the free
central limit theorem asserts.

When $d\geq2$, by combining Theorem \ref{MOOfree} with the main
finding of \cite{KNPS}, we will prove the following free counterpart of
Theorem \ref{NPR}.
%
%
\begin{thmm}\label{NPRfree}
Let $(\mathcal{A},\varphi)$ be a non-commutative probability space.
Let $S_1,S_2,\ldots$ be a sequence of free $\mathcal{S}(0,1)$ random
variables.
Fix $d\geq2$ and consider a sequence of functions $f_N\dvt\{1,\ldots
,N\}^d\to\R$
satisfying the following three assumptions for each $N$ and each
$i_1,\ldots,i_d=1,\ldots,N$:\vadjust{\goodbreak}
\begin{enumerate}[(iii)]
\item[(i)] {(full symmetry)} $ f_N(i_1,\ldots,i_d)=f_N(i_{\sigma
(1)},\ldots,i_{\sigma(d)})$
for all $\sigma\in\mathfrak{S}_d$;
\item[(ii)] {(vanishing on diagonals)} $ f_N(i_1,\ldots,i_d)=0$ if
$i_k=i_l$ for some $k\neq l$;
\item[(iii)] {(normalization)} $ \sum_{j_1,\ldots,j_d=1}^N
f_N(j_1,\ldots,j_d)^2=1$.
\end{enumerate}
Also, let
$Q_N(x_1,\ldots,x_N)$ be the polynomial in non-commuting variables
given by (\ref{qnfree}).
Then, the following two conclusions are equivalent as $N\to\infty$:
\begin{enumerate}[(A)]
\item[(A)] $Q_N(S_1,\ldots,S_N) \stackrel{\mathrm{law}}{\to} \mathcal{S}(0,1)$;
\item[(B)] $Q_N(X_1,\ldots,X_N) \stackrel{\mathrm{law}}{\to} \mathcal
{S}(0,1)$ for any sequence $X_1,X_2,\ldots$ of free identically
distributed and centered random variables with unit variance.
\end{enumerate}
\end{thmm}

Although a weak `mirror-symmetry' assumption would have been
undoubtedly more natural, we impose in Theorem \ref{NPRfree} the same
`full symmetry' assumption (i) than in Theorem \ref{NPR}. This is
unfortunately not insignificant in our non-commutative framework.
But we cannot expect better by using our strategy of proof, as is
illustrated by a concrete counterexample
in Section \ref{secuniversality}.

Theorem \ref{NPRfree} may be seen as
a free universality phenomenon, in the sense that the semicircular
behavior of $Q_N(X_1,\ldots,X_N)$ is
asymptotically insensitive to the distribution of its summands.
In reality, this is more subtle, as the following explicit situation
well illustrates in
the case $d=2$ (quadratic case). Indeed, let us consider
\[
Q_N(x_1,\ldots,x_N)=\frac{1}{\sqrt{2N-2}}\sum
_{i=2}^N (x_1x_i+x_ix_1),\qquad
N\geq2,
\]
let $S_1,S_2,\ldots$ be a sequence of free $\mathcal{S}(0,1)$ random
variables and
let $X_1,X_2,\ldots$ be a sequence of free Rademacher\vspace*{-2pt} random variables
(i.e., the law of $X_1$ is given by
$\frac12 \delta_1 + \frac12\delta_{-1}$).
Then
$Q_N(X_1,\ldots,X_N)\stackrel{\mathrm{law}}{\to} \mathcal{S}(0,1)$ as
$N\to
\infty$,
but
\[
Q_N(S_1,\ldots,S_N)\stackrel{\mathrm{law}} {\to}
\frac{1}{\sqrt{2}}(S_1S_2+S_2S_1)
\not\sim\mathcal{S}(0,1).
\]
(See Section \ref{secuniversality} for the details.)
This means that it is possible to have $Q_N(X_1,\ldots,X_N)$ converging
in law to $\mathcal{S}(0,1)$
for a \textit{particular} centered distribution of $X_1$, without
having the same phenomenon for \textit{every} centered distribution
with variance one. The question of which are the distributions that
enjoy such a universality phenomenon is still an open problem. (In the
commutative case, it is known that the Gaussian and the Poisson
distributions both lead to universality, see \cite{NPR,PZ}. Yet there
are no other examples.)

\textit{Organization of the paper}.
The rest of our paper is organized as follows. In Section \ref
{secuniversality}, we deduce from Theorem \ref{MOOfree} several
results connected with the universality phenomenon and we study the
limitations of Theorem \ref{NPRfree}.
Section \ref{secMOOfree} is devoted to the proof of Theorem~\ref{MOOfree}.

\section{Free universality}\label{secuniversality}

In this section, we show how Theorem \ref{MOOfree} leads to several
results connected with the universality phenomenon. We also study the
limitations of Theorem \ref{NPRfree}:
Can we replace the role played by the semicircular distribution by any
other law? Can we replace the full symmetry
assumption (i) by a more natural one?

To do so, we first need to recall some facts proven in references \cite
{bianespeicher,KNPS}.

\textit{Convergence of Wigner integrals}. For $1\leq p \leq\infty$,
we write $L^p(\mathcal{A},\ff)$ to indicate the $L^p$ space obtained as
the completion of
$\mathcal{A}$ with respect to the norm $\| A\|_p = \ff(|A|^p)^{1/p}$,
where $ |A|= \sqrt{A^\ast A}$, and $\|\cdot\|_\infty$ stands for the
operator norm.
For every integer $q\geq2$, the space $L^2(\mathbb{R}_+^q)$ is the
collection of all real-valued
functions on $\mathbb{R}_+^q$ that are square-integrable with respect
to the Lebesgue measure.
Given $f\in L^2(\mathbb{R}_+^q)$, we write
$f^*(t_1,t_2,\ldots,t_q) = f(t_q,\ldots,t_2,t_1)$,
and we call $f^*$ the \textit{adjoint} of $f$. We say that an element
of $L^2(\mathbb{R}_+^q)$ is \textit{mirror symmetric} whenever
$f = f^*$ as a function.
Given $f\in L^2(\mathbb{R}_+^q)$ and $g\in L^2(\mathbb{R}_+^p)$, for
every $r = 1,\ldots,p\wedge q$
we define the $r$th \textit{contraction} of $f$ and $g$ as the element
of $L^2(\mathbb{R}_+^{p+q-2r})$ given by
%
%
\begin{eqnarray}
\label{contr} &&f{\mathop{\frown}^{r}} g (t_1,
\ldots,t_{p+q-2r})
\nonumber
\\[-8pt]
\\[-8pt]
\nonumber
&&\quad= \int_{\R_+^{p+q-2r}} f(t_1,\ldots,t_{p-r},x_1,
\ldots,x_r)g(x_r,\ldots,x_1,t_{p-r+1},
\ldots,t_{p+q-2r}) \,\mathrm{d}x_1\cdots\, \mathrm{d}x_r.
\end{eqnarray}
One also writes $f{\mathop{\frown}^{0}} g (t_1,\ldots,t_{p+q}) =
f\otimes g
(t_1,\ldots,t_{p+q}) = f(t_1,\ldots,t_q)g(t_{q+1},\ldots,t_{p+q})$. In the
following, we shall use the notation $f{\mathop{\frown}^{0}} g$ and
$f\otimes g$
interchangeably. Observe that, if $p=q$, then $f{\mathop{\frown}^{p}}
g = \langle
f,g^{*}\rangle_{L^2(\R_+^q)}$.

A \textit{free Brownian motion} $S$ on $(\mathcal{A},\ff)$ consists of:
(i) a filtration $\{\mathcal{A}_t \dvt t\geq0\}$ of von Neumann
sub-algebras of $\mathcal{A}$ (in particular, $\mathcal{A}_u \subset
\mathcal{A}_t$ for $0\leq u<t$),
(ii) a collection $S = (S_t)_{t\geq0}$ of self-adjoint operators such that:
\begin{itemize}
\item[--] $S_t\in\mathcal{A}_t$ for every $t$;
\item[--] for every $t$, $S_t$ has a semicircular distribution
$\mathcal
{S}(0,t)$;
\item[--] for every $0\leq u<t$, the increment $S_t - S_u$ is freely
independent of $\mathcal{A}_u$, and has a semicircular distribution
$\mathcal{S}(0,t-u)$.
\end{itemize}

For every integer $q\geq1$, the collection of all random variables of
the type $I_q(f)$,
$f \in L^2(\mathbb{R}_+^q)$, is called the $q$th \textit{Wigner chaos}
associated with $S$, and is defined according to
\cite{bianespeicher}, Section 5.3, namely:
\begin{itemize}
\item[--] first define $I_q(f) = (S_{b_1} - S_{a_1})\cdots(S_{b_q} -
S_{a_q})$ for every function $f$ having the form
%
%
\begin{equation}
\label{esimple} f(t_1,\ldots,t_q) = \mathbf{1}_{(a_1,b_1)}(t_1)
\times\cdots\times\mathbf{1}_{(a_q,b_q)}(t_q),
\end{equation}
where the intervals $(a_i,b_i)$, $i=1,\ldots,q$, are pairwise disjoint;
\item[--] extend linearly the definition of $I_q(f)$ to simple
functions vanishing on diagonals, that is, to functions $f$ that are finite
linear combinations of indicators of the type~(\ref{esimple});
\item[--] exploit the isometric relation
%
%
\begin{equation}
\label{efreeisometry} \bigl\langle I_q(f_1),I_q(f_2)
\bigr\rangle_{L^2(\mathcal{A},\ff)}= \varphi \bigl(I_q(f_1)^*I_q(f_2)
\bigr)= \varphi \bigl(I_q \bigl(f_1^*
\bigr)I_q(f_2) \bigr)= \langle f_1,f_2
\rangle_{L^2(\mathbb{R}_+^q)},
\end{equation}
where $f_1,f_2$ are simple functions vanishing on diagonals, and use a
density argument to define $I_q(f)$ for a general
$f\in L^2(\mathbb{R}_+^q)$.
\end{itemize}

Observe that relation (\ref{efreeisometry}) continues to hold for
every pair $f_1,f_2 \in L^2(\mathbb{R}_+^q)$. Moreover, the above sketched
construction implies that $I_q(f)$ is self-adjoint if and only if $f$
is mirror symmetric. We recall the following fundamental
multiplication formula, proven in \cite{bianespeicher}. For every
$f\in
L^2(\R_+^p)$ and $g\in L^2(\R_+^q)$, where $p,q\geq1$, we have
%
%
\begin{equation}
\label{emult} I_p(f)I_q(g) = \sum
_{r=0}^{p\wedge q} I_{p+q-2r}(f{\mathop{
\frown}^{r}}g).
\end{equation}

Let $S_1, S_2,\ldots\sim\mathcal{S}(0,1)$ be freely independent, fix
$d\geq2$, and consider a sequence of functions $f_N\dvt\{1,\ldots,N\}^d\to
\R$
satisfying assumptions (ii) and (iii) of Theorem \ref{NPRfree} as
well as
%
%
\begin{equation}
\label{mirrorfN} f_N(i_1,\ldots,i_d)=f_N(i_d,
\ldots,i_1) \qquad\mbox{for all }N\geq1\mbox{ and }i_1,
\ldots,i_d\in\{1,\ldots,N\}.
\end{equation}
Let also $Q_N(x_1,\ldots,x_N)$ be the polynomial in non-commuting
variables given by (\ref{qnfree}).
Set $e_i=\mathbf{1}_{[i-1,i]}\in L^2(\R_+)$, $i\geq1$. For each $N$,
one has
%
%
\begin{equation}
\label{qNdiff} Q_N(S_1,\ldots,S_N)
\stackrel{ \rm law} {=}Q_N \bigl(I_1(e_1),
\ldots,I_1(e_N) \bigr).
\end{equation}
By applying the multiplication formula (\ref{emult}) and by taking
into account assumption (ii), it is straightforward to check that
%
%
\begin{equation}
\label{representation} Q_N \bigl(I_1(e_1),
\ldots,I_1(e_N) \bigr) =I_d(g_N),
\end{equation}
where
%
%
\begin{equation}
\label{g-n} g_N = \sum_{i_1,\ldots,i_d=1}^N
f_N(i_1,\ldots,i_d) e_{i_1}\otimes
\cdots\otimes e_{i_d}.
\end{equation}
The function $g_N$ is mirror-symmetric (due to (\ref{mirrorfN})) and
has an $L^2(\R_+^d)$-norm equal to 1 (due to (iii)).
Using both Theorems 1.3 and 1.6 of \cite{KNPS} (see also \cite
{yetanother}), we deduce that the
following equivalence holds true as $N\to\infty$:
%
%
\begin{eqnarray}
\label{4thmomentfree} Q_N(S_1,\ldots,S_N)
\stackrel{\mathrm{law}} {\to} \mathcal{S}(0,1) \Longleftrightarrow\|g_N {
\mathop{\frown}^{r}} g_N\|_{L^2(\R_+^{2d-2r})}\to0
\nonumber
\\[-8pt]
\\[-8pt]
\eqntext{\mbox{for
all $r\in\{ 1,\ldots,d-1 \}$}.\qquad\quad}
\end{eqnarray}
%
For $r=d-1$, observe that
%
%
\begin{eqnarray}\label{influ1}
&&\|g_N{\mathop{\frown}^{d-1}} g_N
\|_{L^2(\R_+^{2})}\nonumber\\
&&\quad= \Biggl\Vert \sum_{i,j=1}^N
\Biggl(\sum_{k_2,\ldots,k_d=1}^N f_N(i,k_2,
\ldots,k_d)f_N(k_d,\ldots,k_2,j)
\Biggr) e_{i}\otimes e_{j}\Biggr\Vert_{L^2(\R_+^{2})}
\nonumber
\\
&&\quad= \sqrt{\sum_{i,j=1}^N \Biggl(
\sum_{k_2,\ldots,k_d=1}^N f_N(i,k_2,
\ldots,k_d)f_N(k_d,\ldots,k_2,j)
\Biggr)^2}
\\
&& \quad\geq\sqrt{\sum_{i=1}^N
\Biggl(\sum_{k_2,\ldots,k_d=1}^N f_N(i,k_2,
\ldots,k_d)^2 \Biggr)^2} \qquad\bigl(\mbox{by
setting $j=i$ and using (\ref{mirrorfN})} \bigr)
\nonumber
\\
&&\quad\geq \max_{i=1,\ldots,N}\sum_{k_2,\ldots,k_d=1}^N
f_N(i,k_2,\ldots,k_d)^2.\nonumber
\end{eqnarray}
%

\begin{pf*}{Proof of Theorem \ref{NPRfree}}
Of course, only the
implication $\mathrm{(A)}\to\mathrm{(B)}$ has to be shown.
Assume that $\mathrm{(A)}$ holds. Then, using (\ref{4thmomentfree}) (condition
(i) implies in particular (\ref{mirrorfN})), we get that
$\|g_N{\mathop{\frown}^{d-1}} g_N\|_{L^2(\R_+^{2})}\to0$ as $N\to
\infty$.
Using (\ref{influ1}) and since $f_N$ is fully-symmetric, we deduce that
the quantity $\tau_N$ of Theorem \ref{MOOfree} tends to zero as $N$
goes to infinity. This, combined with assumption \textup{(A)} and (\ref
{MOOfree-eq}), leads to $\mathrm{(B)}$.
\end{pf*}

\textit{A counterexample}. In Theorem \ref{NPRfree}, can we replace the
role played by the semicircular distribution
by any other law? The answer is no in general. Indeed, let us take a
look at the following situation.
Fix $d=2$ and consider
\[
Q_N(x_1,\ldots,x_N)=\frac{1}{\sqrt{2N-2}}\sum
_{i=2}^N (x_1x_i+x_ix_1),\qquad
N\geq2.
\]
Let $S_1,S_2,\ldots$ be a sequence of free $\mathcal{S}(0,1)$ random
variables and
let $X_1,X_2,\ldots$ be a sequence of free Rademacher random variables
(i.e., the law of $X_1$ is
given by
$\frac12 \delta_1 + \frac12\delta_{-1}$).
Then, using the free central limit theorem, it is clear on one hand that
\begin{eqnarray*}
Q_N(X_1,\ldots,X_N) &=&\frac{1}{\sqrt{2}}
X_1 \Biggl(\frac{1}{\sqrt{N-1}}\sum_{i=2}^N
X_i \Biggr) + \frac{1}{\sqrt{2}} \Biggl(\frac{1}{\sqrt{N-1}}\sum
_{i=2}^N X_i
\Biggr)X_1
\\
& \stackrel{\mathrm{law}} {\to}& \frac{1}{\sqrt{2}} (X_1 S_1
+ S_1 X_1 )\qquad \mbox{as $N\to\infty$},
\end{eqnarray*}
with $X_1$ and $S_1$ freely independent.
By Proposition 1.10 and identity (1.10) of Nica and Speicher~\cite
{NicaSpeicherDuke}, it turns out
that $\frac{1}{\sqrt{2}} (X_1 S_1 + S_1 X_1 )\sim\mathcal{S}(0,1)$.
But, on the other hand,
\begin{eqnarray*}
Q_N(S_1,\ldots,S_N) &=&\frac{1}{\sqrt{2}}
S_1 \Biggl(\frac{1}{\sqrt{N-1}}\sum_{i=2}^N
S_i \Biggr) + \frac{1}{\sqrt{2}} \Biggl(\frac{1}{\sqrt{N-1}}\sum
_{i=2}^N S_i
\Biggr)S_1
\\
&\stackrel{\mathrm{law}} {=}& \frac{1}{\sqrt{2}} (S_1 S_2 +
S_2 S_1 ).
\end{eqnarray*}
The random variable $\frac{1}{\sqrt{2}} (S_1 S_2 + S_2 S_1 )$
being \textit{not} $\mathcal{S}(0,1)$ distributed
(its law is indeed the so-called \textit{tetilla law}, see \cite
{tetilla}), we deduce that one cannot replace
the role played by the semicircular distribution in Theorem \ref
{NPRfree} by the Rademacher distribution.

\textit{Another counterexample}. In Theorem \ref{NPRfree}, can we
replace the full symmetry assumption (i) by the mirror-symmetry
assumption? Unfortunately, we have not been able to answer this
question. But if the answer is yes, what is sure is that we cannot use
the same arguments as in the fully-symmetric case to show such a
result. Indeed,
when $f_N$ is fully-symmetric we have
\[
\tau_N = d \times\max_{i=1,\ldots,N}\sum
_{k_2,\ldots,k_d=1}^N f_N(i,k_2,
\ldots,k_d)^2,
\]
allowing us to prove Theorem \ref{NPRfree} by using the following set
of implications: as $N\to\infty$,
%
%
\begin{eqnarray}\label{implication}
Q_N(S_1,\ldots,S_N) \stackrel{\mathrm{law}} {\to} \mathcal{S}(0,1) &\stackrel{(\ref{4thmomentfree})} {\Longrightarrow}&
\|g_N{\mathop{\frown}^{d-1}} g_N
\|_{L^2(\R_+^{2})}\to0 \stackrel{( \ref{influ1})} {\Longrightarrow} \tau_N
\to0
\nonumber
\\[-8pt]
\\[-8pt]
\nonumber
&\stackrel{\mathrm{Theorem\ }\ref{MOOfree}} {\Longrightarrow} & Q_N(X_1,
\ldots,X_N) \stackrel{\mathrm{law}} {\to} \mathcal{S}(0,1)
.
\end{eqnarray}
Unfortunately, when $f_N$ is only mirror-symmetric the implication
%
%
\begin{equation}
\label{fausseimpl} \|g_N{\mathop{\frown}^{d-1}}
g_N\|_{L^2(\R_+^{2})} \to0\Longrightarrow\tau_N \to0,
\end{equation}
that plays a crucial role in (\ref{implication}), is no longer true in general.
To see why, let us consider the following counterexample (for which we
fix $d=3$). Define first a sequence of functions $f'_N\dvt\{1,\ldots
,N\}^2
\to\R$ according to the formula
\[
f_N'(i,i+1)=f_N'(i+1,i)=
\frac{1}{\sqrt{2N-2}},
\]
and $f_N'(i,j)=0$ whenever $i=j$ or $|j-i|\geq2$. Next, for $i,j,k \in
\{1,\ldots,N\}$, set
%
%
\begin{eqnarray}
\label{count-f-n} &&f_N(i,j,k)
\nonumber
\\[-8pt]
\\[-8pt]
\nonumber
&&\quad= \cases{ %
0, & \quad $\mbox{if } j\geq2 \mbox{ or } (j=1 \mbox{ and } i=1) \mbox{ or }
(j=1 \mbox{ and } k=1),$
\vspace*{2pt}\cr
f'_{N-1}(i-1,k-1), &\quad  $\mbox{otherwise}.$}
\end{eqnarray}
Easy-to-check properties of $f_N$ include mirror-symmetry, vanishing on
diagonals property,
\[
\sum_{i,j,k=1}^N f_N(i,j,k)^2
= \sum_{i,k=1}^{N-1} f'_{N-1}(i,k)^2=1
\]
and
%
%
\begin{eqnarray}
\label{contract} &&\sum_{i,j=1}^N \Biggl(
\sum_{k,l=1}^N f_N(i,k,l)
f_N(l,k,j) \Biggr)^2
\nonumber
\\[-8pt]
\\[-8pt]
\nonumber
&&\quad= \sum_{i,j=1}^N
\Biggl( \sum_{l=1}^{N-1} f'_{N-1}(i,l)
f'_{N-1}(l,j) \Biggr)^2\to0.
\end{eqnarray}
Let $g_N$ be given by (\ref{g-n}), that is,
\[
g_N = \frac{1}{\sqrt{2N-4}}\sum_{i=1}^{N-2}
(e_{i+1}\otimes e_1\otimes e_{i+2} +
e_{i+2}\otimes e_1\otimes e_{i+1} ).
\]
The limit (\ref{contract}) can be readily translated into
$\|g_N {\mathop{\frown}^{2}} g_N\|^2_{L^2(\R_+^2)} \to0$ as $N \to
\infty$.
On the other hand, we have
\begin{eqnarray*}
\tau_N&=&\max_{1\leq j\leq N} \operatorname{Inf}_j(f_N)=
\max_{1\leq
j\leq N} \sum_{i,k=1}^N \bigl\{
f_N(i,j,k)^2+f_N(j,i,k)^2+f_N(i,k,j)^2
\bigr\}
\\
&\geq&\max_{1\leq j\leq N} \sum_{i,k=1}^N
f_N(i,j,k)^2 = \sum_{i,k=1}^N
f_N(i,1,k)^2 = 1,
\end{eqnarray*}
which contradicts (\ref{fausseimpl}), as announced.

It is also worth noting that the sequence of functions $f_N$ defined by
(\ref{count-f-n}) provides an explicit counterexample to the so-called
\emph{Wiener-Wigner transfer principle} (see \cite{KNPS}, Theorem 1.8)
in a non-fully-symmetric situation. Indeed, on one hand, we have
\[
\|g_N {\mathop{\frown}^{1}} g_N
\|^2_{L^2(\R_+^2)}= \|g_N {\mathop{\frown}^{2}}
g_N\|^2_{L^2(\R
_+^2)} \to0 \qquad\mbox{as } N \to\infty,
\]
which, due to (\ref{4thmomentfree}), entails that $Q_N(S_1,\ldots,S_N)
\stackrel{\mathrm{law}}{\to} \mathcal{S}(0,1)$.
On the other hand,
let $G_1,\ldots,G_N \sim\mathcal{N}(0,1)$ be independent random
variables defined on a (classical) probability space $(\Omega,\mathcal
{F},P)$. One has
\[
Q_N(G_1,\ldots,G_N)=G_1\times
\Biggl( \frac{2}{\sqrt{2N-4}} \sum_{i=2}^{N-1}
G_{i}G_{i+1} \Biggr),
\]
and it is easily checked that $\frac{2}{\sqrt{2N-4}} \sum_{i=2}^{N-1}
G_{i}G_{i+1}\stackrel{\mathrm{law}}{\to} \mathcal{N}(0,2)$ (apply, e.g., the
Fourth Moment theorem of \cite{NP}). As a result, the sequence
$Q_N(G_1,\ldots,G_N)$ converges in law to $\sqrt{2} G_1G_2$,
which is not Gaussian. This leads to our desired contradiction.


\textit{Free CLT for homogeneous sums}. As an application of Theorem
\ref{MOOfree}, let us also highlight the following practical
convergence criterion for multilinear polynomials, which can be readily
derived from (\ref{4thmomentfree}).

%
\begin{thmm}
Let $(\mathcal{A},\varphi)$ be a non-commutative probability space.
Let $X_1,X_2,\ldots$ be a sequence of centered free random variables
with unit variance satisfying
$\sup_{i\geq1}\ff(|X_i|^r)<\infty$ for all $r\geq1$.
Fix $d\geq1$, and consider a sequence of functions $f_N\dvt\{1,\ldots
,N\}^d\to\R$
satisfying the three basic assumptions \textup{(i)--(ii)--(iii)} of Theorem \ref
{MOOfree}. Assume moreover that, as $N$ tends to infinity, $\max_{1\leq
j\leq N} \operatorname{Inf}_j(f_N)\to0$ and $\|g_N{\mathop{\frown}^{r}} g_N\|_{L^2(\R
_+^{2d-2r})}\to0$ for all $r\in\{1,\ldots,d-1\}$, where $g_N$ is
defined through (\ref{g-n}).
Then one has
%
%
\begin{equation}
\sum_{i_1,\ldots,i_d=1}^N
f_N(i_1,\ldots,i_d) X_{i_1}\cdots
X_{i_d}\stackrel{\mathrm{law}} {\to} \mathcal{S}(0,1).
\end{equation}
\end{thmm}

For instance, thanks to this result one can easily check that, given a
positive integer $k$, one has
\[
\frac{1}{\sqrt{N}} \sum_{i=1}^{N-k} \{
X_i X_{i+1}\cdots X_{i+k}+X_{i+k}X_{i+k-1}
\cdots X_i \}\stackrel{\mathrm{law}} {\to} \mathcal{S}(0,1) \qquad\mbox{as $N
\to\infty$}
\]
for any sequence $(X_i)$ of centered free random variables with unit
variance satisfying
$\sup_{i\geq1}\ff(|X_i|^r)<\infty$ for all $r\geq1$.

\section{\texorpdfstring{Proof of Theorem \protect\ref{MOOfree}}
{Proof of Theorem 1.3}}\label{secMOOfree}

As in \cite{MOO}, our strategy is essentially based on a generalization
of the classical Lindeberg method, which was originally designed for
linear sums of (classical) random variables (see \cite{lin}). Before we
turn to the details of the proof, let us briefly report the two main
differences with the arguments displayed in \cite{MOO} for commuting
random variables.

First, in this non-commutative context, we can no longer rely on some
classical Taylor expansion as a starting point of our study. This issue
can be easily overcome though, by resorting to abstract expansion
formulae (see~(\ref{bino})) together with appropriate H\"older-type
estimates (see~(\ref{holder})). As far as this particular point is
concerned, the situation is quite similar to what can be found in~\cite
{kargin}, even if the latter reference is only concerned with the
linear case, that is, $d=1$.

Another additional difficulty raised by this free background lies in
the transposition of the hypercontractivity property, which is at the
core of the procedure.
In \cite{MOO}, the proof of hypercontractivity for multilinear
polynomials heavily depends on the fact that the variables do commute
(see, e.g., the proof of \cite{MOO}, Proposition 3.11). Hence, new
arguments are needed here
and we postpone this point to Section \ref{subsechyper}.

\subsection{General strategy}
For the rest of the section, we fix two sequences $(X_i), (Y_i)$ of
random variables in a non-commutative probability space $(\mathcal
{A},\varphi)$, two integers $N,m\geq1$, as well as a function
$f_N\dvt\{
1,\ldots,N\}^d \to\R$ giving rise to a polynomial $Q_N$ through
(\ref
{qn}), and we assume that all of these objects meet the requirements of
Theorem \ref{MOOfree}. In accordance with the Lindeberg method, we are
first prompted to introduce some additional notation.

\textit{Notation}.
For every $i\in\{1,\ldots,N+1\}$, let us consider the vector
\[
Z^{N,(i)}:=(Y_1,\ldots,Y_{i-1},X_{i},
\ldots,X_N).
\]
In particular, $Z^{N,(1)}=(X_1,\ldots,X_N)$ and
$Z^{N+1,(N)}=(Y_1,\ldots
,Y_N)$, so that
%
%
\begin{equation}
\label{decom} Q_N(X_1,\dots,X_N)^m-Q_N(Y_1,
\ldots,Y_N)^m=\sum_{i=1}^{N}
\bigl[Q_N \bigl(Z^{N,(i)} \bigr)^m-Q_N
\bigl(Z^{N,(i+1)} \bigr)^m \bigr].
\end{equation}
Since the only difference between the vectors $Z^{N,(i)}$ and
$Z^{N,(i+1)}$ is their $i$th-component, it is readily checked that
\[
Q_N \bigl(Z^{N,(i)} \bigr)=U_N^{(i)}+V_N^{(i)}(X_i)
\quad\mbox{and}\quad Q_N \bigl(Z^{N,(i+1)} \bigr)=U_N^{(i)}+V_N^{(i)}(Y_i)
,
\]
where $U^{(i)}_N$ stands for the multilinear polynomial
\[
U^{(i)}_N:= \sum_{j_1,\ldots,j_d \in\{1,\ldots,N\} \setminus\{i\}}
f_N(j_1,\ldots,j_d) Z^{N,(i)}_{j_1}
\cdots Z^{N,(i)}_{j_d},
\]
and $V_N^{(i)}\dvt\ca\to\ca$ is the linear operator defined, for every
$x\in\ca$, by
\begin{eqnarray*}
&&V_N^{(i)}(x)
\\
&&\quad:=\sum_{l=1}^d \sum
_{j_1,\ldots,j_{d-1} \in\{1,\ldots,N\}\setminus\{i\}
} f_N(j_1,\ldots,j_{l-1},i,j_l
\ldots,j_{d-1}) Z^{N,(i)}_{j_1} \cdots
Z^{N,(i)}_{j_{l-1}} x Z^{N,(i)}_{j_l} \cdots
Z^{N,(i)}_{j_{d-1}}.
\end{eqnarray*}

\textit{Expansion}. Once endowed with the above notation, the problem
reduces to examining the differences
%
%
\begin{equation}
\label{diff} \ff \bigl( \bigl(U^{(i)}_N+V^{(i)}_N(X_i)
\bigr)^m \bigr)- \ff \bigl( \bigl(U^{(i)}_N+V^{(i)}_N(Y_i)
\bigr)^m \bigr)
\end{equation}
for $i\in\{1,\ldots,N-1\}$. In a commutative context, this could be
handled with the classical binomial formula. Although such a mere
formula is not available here, one can still assert that for every
$A,B\in\ca$,
%
%
\begin{equation}
\label{bino} (A+B)^m=A^m+\sum
_{n=1}^m \sum_{(r,\mathbf{i}_{r+1},\mathbf{j}_r)\in\cd
_{m,n}}
c_{m,n,r,\mathbf{i}_{r+1},\mathbf{j}_r} A^{i_1} B^{j_1} A^{i_2}
B^{j_2} \cdots A^{i_r} B^{j_r} A^{i_{r+1}},
\end{equation}
where
\[
\cd_{m,n}:= \Biggl\{(r,\mathbf{i}_{r+1},\mathbf{j}_r)
\in\{1,\ldots,m\} \times\N^{r+1}\times\N^r\dvt\sum
_{l=1}^{r+1} i_l=n , \sum
_{l=1}^{r} j_l=m-n \Biggr\}
\]
and the $c_{m,n,r,\mathbf{i}_{r+1},\mathbf{j}_r}$'s stand for
appropriate combinatorial coefficients (independent of $A$ and $B$).
The sets $\cd_{m,n}$ must of course be understood as follows: given
$(r,\mathbf{i}_{r+1},\mathbf{j}_r)\in\cd_{m,n}$, the product $ A^{i_1}
B^{j_1} A^{i_2} B^{j_2} \cdots A^{i_r} B^{j_r} A^{i_{r+1}}$ contains
$A$ exactly $n$ times
and $B$ exactly $(m-n)$ times, both counted with multiplicity.

Let us go back to (\ref{diff}) and let us apply formula (\ref{bino}) in
order to expand $(U^{(i)}_N+V^{(i)}_N(X_i))^m$ (resp.,
$(U^{(i)}_N+V^{(i)}_N(Y_i))^m$). The first and second order terms
(i.e., for $n=1,2$ in (\ref{bino})) of the resulting sum happen to
vanish, as a straightforward use of the following lemma shows.

%
\begin{lemma}\label{condi}
Let $Y$ and $Z$ be two centered random variables with unit variance.
Then, for every integer $k\geq1$ and every sequence $(X_i)$ of
centered freely independent random variables independent of $Y$ and
$Z$, one has
%
%
\begin{equation}
\label{ordre-un} \varphi( X_{i_1} \cdots X_{i_r} Y
X_{i_{r+1}} \cdots X_{i_k} )=\varphi( X_{i_1} \cdots
X_{i_r} Z X_{i_{r+1}} \cdots X_{i_k} )=0
\end{equation}
and
%
%
\begin{equation}
\label{ordre-deux}\hspace*{-6pt} \varphi( X_{i_1} \cdots X_{i_r} Y
X_{i_{r+1}} \cdots X_{i_s} Y X_{i_{s+1}} \cdots
X_{i_k} )=\varphi( X_{i_1} \cdots X_{i_r} Z
X_{i_{r+1}} \cdots X_{i_s} Z X_{i_{s+1}} \cdots
X_{i_k} )
\end{equation}
for all $0\leq r\leq s \leq k$ and $(i_1,\ldots,i_k)\in\N^k$.
\end{lemma}

\begin{pf}
Let us first focus on (\ref{ordre-un}). For $k=1$, this is obvious.
Assume that the result holds true up to $k-1$ and write
\[
\varphi( X_{i_1} \cdots X_{i_r} Y X_{i_{r+1}} \cdots
X_{i_k} )=\varphi \bigl( X_{i_1'}^{m_{1}} \cdots
X_{i_{r'}'}^{m_{r'}}Y X_{i_{r'+1}'}^{m_{r'+1}} \cdots
X_{i_{s'}'}^{m_{s'}} \bigr)
\]
with $i_{p+1}'\neq i_p'$ for $p\in\{1,\ldots,s'-1\} \setminus\{r'\}
$, $i'_{s'} \neq i'_1$ and $m_p\geq1$ for every $p\in\{1,\ldots,s'\}
$. Center successively every random variable
$X^{m_{p_1}}_{i'_{p_1}},\ldots,X^{m_{p_t}}_{i'_{p_t}}$ for which
$m_{p_i}\geq2$: together with an induction argument, this yields
\begin{eqnarray*}
&&\varphi \bigl( X_{i_1'}^{m_{1}} \cdots X_{i_{r'}'}^{m_{r'}}Y
X_{i_{r'+1}'}^{m_{r'+1}} \cdots X_{i_{s'}'}^{m_{s'}} \bigr)
\\
&&\quad= \varphi \bigl( X_{i_1'}\cdots X_{i_{p_1-1}'} \bigl(X_{i'_{p_1}}^{m_{p_1}}-
\varphi \bigl( X_{i'_{p_1}}^{m_{p_1}} \bigr) \bigr)X_{i_{p_1+1}'}^{m_{p_1+1}}
\cdots X_{i_{r'}'}^{m_{r'}}Y X_{i_{r'+1}'}^{m_{r'+1}} \cdots
X_{i_{s'}'}^{m_{s'}} \bigr)
\\
&&\quad=\varphi \bigl( X_{i_1'}\cdots X_{i_{p_1-1}'} \bigl(X_{i'_{p_1}}^{m_{p_1}}-
\varphi \bigl( X_{i'_{p_1}}^{m_{p_1}} \bigr) \bigr)X_{i_{p_1+1}'}
\cdots X_{i_{p_2-1}'} \bigl(X_{i'_{p_2}}^{m_{p_2}}-\varphi \bigl(
X_{i'_{p_2}}^{m_{p_2}} \bigr) \bigr)
\\
&&\hspace*{30pt}X_{i_{p_2+1}'}^{m_{p_2+1}}\cdots X_{i_{r'}'}^{m_{r'}}Y
X_{i_{r'+1}'}^{m_{r'+1}} \cdots X_{i_{s'}'}^{m_{s'}} \bigr) =
\cdots=0
\end{eqnarray*}
owing to free independence. Identity (\ref{ordre-deux}) can be easily
derived from a similar induction procedure.
\end{pf}

Let us go back to the proof of Theorem \ref{MOOfree}. As a consequence
of the previous lemma, it now suffices to establish that, either for
$W=X_i$ or for $W=Y_i$, one has, as soon as $\sum_l j_l \geq3$,
%
%
\begin{equation}
\label{suffit} \hspace*{-6pt}\bigl|\varphi\bigl( \bigl(U_N^{(i)}
\bigr)^{i_1} \bigl(V^{(i)}_N(W)
\bigr)^{j_1} \bigl(U_N^{(i)} \bigr)^{i_2}
\bigl(V^{(i)}_N(W) \bigr)^{j_2} \cdots
\bigl(U_N^{(i)} \bigr)^{i_r} \bigl(V^{(i)}_N(W)
\bigr)^{j_r} \bigr) \bigr| \leq c_{m,d} \operatorname{Inf}_i
(f_N)^{3/2}
\end{equation}
for some constant $c_{m,d}$. Indeed, in this case, by combining (\ref
{decom}), (\ref{bino}) and (\ref{suffit}) with the identities in the
statement of Lemma \ref{condi}, we get
\begin{eqnarray*}
\bigl|\varphi\bigl(Q_N(X_1,\dots,X_N)^m
\bigr)-\varphi\bigl(Q_N(Y_1,\ldots ,Y_N)^m
\bigr)\bigr|&\leq& C_{m,d} \sum_{i=1}^{N}
\operatorname{Inf}_i (f_N)^{3/2}
\\
&\leq& C_{m,d} \tau_N^{1/2} \sum
_{i=1}^N \operatorname{Inf}_i
(f_N) = C_{m,d} \tau_N^{1/2},
\end{eqnarray*}
which is precisely the expected bound of Theorem \ref{MOOfree}.

In order to prove (\ref{suffit}), let us first resort to the following
H\"older-type inequality, borrowed from~\cite{kargin}, Lemma 12:
%
%
\begin{eqnarray}
\label{holder}
&&\bigl|\varphi \bigl( \bigl(U^{(i)}_N
\bigr)^{i_1} \bigl(V^{(i)}_N(W)
\bigr)^{j_1} \cdots \bigl(U^{(i)}_N
\bigr)^{i_r} \bigl(V^{(i)}_N(W)
\bigr)^{j_r} \bigr)\bigr|
\nonumber\qquad
\\[-8pt]
\\[-8pt]
\nonumber
&&\quad\leq \varphi \bigl( \bigl(U^{(i)}_N \bigr)^{2^ri_1}
\bigr)^{2^{-r}} \varphi \bigl( \bigl(V^{(i)}_N(W)
\bigr)^{2^r j_1} \bigr)^{2^{-r}} \cdots\varphi \bigl(
\bigl(U^{(i)}_N \bigr)^{2^ri_r} \bigr)^{2^{-r}}
\varphi \bigl( \bigl(V^{(i)}_N(W) \bigr)^{2^r
j_r}
\bigr)^{2^{-r}}.\qquad
\end{eqnarray}
Now, let the key (forthcoming) Proposition \ref{prophyp} come into the
picture. Thanks to it, we can simultaneously assert that, for every
$p\geq1$,
\[
\varphi \bigl( \bigl(U^{(i)}_N \bigr)^{2p} \bigr)
\leq C_{p,d} \quad\mbox{and}\quad \varphi \bigl( V_N^{(i)}(X_i)^{2p}
\bigr) \leq C_{p,d} \cdot\operatorname{Inf}_i
(f_N)^p
\]
for some constant $C_{p,d}$. Going back to (\ref{holder}), we deduce
that for every $(j_l)$ such that $\sum_l j_l\geq3$,
\begin{eqnarray*}
\bigl|\varphi\bigl( \bigl(U^{(i)}_N\bigr)^{i_1}
\bigl(V^{(i)}_N(X_i)\bigr)^{j_1} \cdots
\bigl(U^{(i)}_N\bigr)^{i_r} \bigl(V^{(i)}_N(X_i)
\bigr)^{j_r} \bigr)\bigr| &\leq& C'_{r,d} \cdot
\operatorname{Inf}_i(f_N)^{2^{-1} (j_1+\cdots
+j_r)}
\\
&\leq& C'_{r,d} \cdot\operatorname{Inf}_i(f_N)^{3/2}
\end{eqnarray*}
since $\operatorname{Inf}_i(f_N)\leq1$, and so the proof of Theorem
\ref{MOOfree} is done.

\subsection{Hypercontractivity}\label{subsechyper}
In order to prove the forthcoming Proposition \ref{prophyp} (which
played an important role in the proof of Theorem \ref{MOOfree}),
we first need a technical lemma. To state it, a few additional
notation must be introduced.

%
\begin{definition}
Fix integers $n_1,\ldots,n_r \geq1$. Any set of disjoint blocks of
points in $\{1,\ldots,n_1+\cdots+n_r\}$ is called a \emph{graph} of
$\{
1,\ldots,n_1+\cdots+n_r\}$. A graph is \emph{complete} if the union of
its blocks covers the whole set $\{1,\ldots,n_1+\cdots+n_r\}$. Besides,
a graph is said to \emph{respect} $n_1\otimes\cdots\otimes n_r$ if
each of its blocks contains at most one point in each set $\{1,\ldots
,n_1\}$, $\{n_1+1,\ldots,n_2\},\ldots,\{n_1+\cdots
+n_{r-1}+1,\ldots
,n_1+\cdots+n_r\}$.

Finally, we denote by $\mathcal{G}_\ast(n_1\otimes\cdots\otimes n_r)$
the set of graphs respecting $n_1\otimes\cdots\otimes n_r$ and
containing no singleton (i.e., no block with exactly one element), and
by $\mathcal{G}_\ast^c(n_1\otimes\cdots\otimes n_r)$ the subset of
complete graphs in $\mathcal{G}_\ast(n_1\otimes\cdots\otimes n_r)$.
\end{definition}

Now, given a graph $\ga$ of $\{1,\ldots,n\}$ with $p$ vertices
($p\leq
n$) and a function $f\dvt\{1,\ldots,\break N\}^n \to\R$, we call \emph
{contraction} of $f$ with respect to $\ga$ the function $C_\ga(f)\dvt
\{
1,\ldots,N\}^{n-p} \to\R$ defined for every $(j_1,\ldots,j_{n-p})$ by
the formula
\begin{eqnarray*}
&&C_\ga(f) (j_1,\ldots,j_{n-p})\\
&&\quad:= \sum
_{i_1,\ldots,i_p=1}^N f(j_1,\ldots
,i_1,\ldots,i_p,\ldots,j_{n-p})\cdot\delta(
\ga, j_1,\ldots,i_1,\ldots,i_p,
\ldots,j_{n-p}),
\end{eqnarray*}
where:
\begin{itemize}
\item the (fixed) positions of the $i_k$'s in $(j_1,\ldots
,i_1,\ldots,i_p,\ldots,j_{n-p})$ correspond to the positions of the
vertices of $\ga$;

\item $\delta( \ga,j_1,\ldots,i_1,\ldots,i_p,\ldots,j_{n-p}
)=1$ if all $i_k,i_l$ in a same block of $\ga$ are equal, and $0$ otherwise.
\end{itemize}

With these notation in hand, we can prove the following lemma.

%
\begin{lemma}\label{lemtechn}
For every $\ga\in\mathcal{G}_\ast(n_1 \otimes\cdots\otimes n_r)$
and all $f_i \in\ell^2(\{1,\ldots,N\}^{n_i})$ ($i=1,\ldots,r$), one has
\[
\bigl\| C_\ga(f_1 \otimes\cdots\otimes f_r)
\bigr\|_{\ell^2} \leq\prod_{i=1}^r
\|f_i\|_{\ell^2}.
\]
\end{lemma}

\begin{pf}
We use an induction procedure on $r$. When $r=1$, $C_\ga(f_1)=f_1$. Fix
now $r\geq2$ and $\ga\in\mathcal{G}_\ast(n_1 \otimes\cdots
\otimes
n_r)$. Denote by $\tilde{\ga} \in\mathcal{G}_\ast(n_2 \otimes
\cdots
\otimes n_r)$ the restriction of $\ga$ to $n_2 \otimes\cdots\otimes
n_r$ (i.e., the graph that one obtains from $\ga$ by getting rid of
the blocks with vertices in $\{1,\ldots,n_1\}$). If $\ga$ has no vertex
in $\{1,\ldots,n_1\}$, then
\[
C_\ga(f_1 \otimes\cdots\otimes f_r)=f_1
\otimes C_{\tilde{\ga
}}(f_2\otimes\cdots\otimes f_r)
\]
and we can conclude by induction. Otherwise, it is easily seen that
$ \| C_\ga(f_1 \otimes\cdots\otimes f_r) \|_{\ell^2}^2$ can be
decomposed as
\begin{eqnarray*}
&&\bigl\| C_\ga(f_1 \otimes\cdots\otimes f_r)
\bigr\|_{\ell^2}^2\\
&&\quad =\sum_{i_1,\ldots,i_l,j_1,\ldots,j_m} \biggl(
\sum_{k_1,\ldots
,k_q}f_1(i_1,
\ldots,k_1,\ldots,k_q,\ldots,i_l)\\
&&\qquad\hspace*{83pt}{}\times
C_{\tilde{\ga}}(f_2 \otimes\cdots\otimes f_r)
(j_1,\ldots, k_{\si(1)},\ldots,k_{\si
(p)},
\ldots,j_m) \biggr)^2,
\end{eqnarray*}
where:
\begin{itemize}
\item $l$ (resp., $m$) is the number of points in $\{1,\ldots
,n_1\}$
(resp., $\{n_1+1,\ldots,n_1+\cdots+n_r\}$) which are not assigned by
$\ga$;

\item in $f_1(i_1,\ldots,k_1,\ldots,k_q,\ldots, i_l)$, the (fixed)
positions of the $k_i$'s correspond to the positions of the $q$
vertices of $\ga$ in $\{1,\ldots,n_1\}$;

\item $\si\dvt\{1,\ldots,p\} \to\{1,\ldots,q\}$ ($p\geq q$)
is a
surjective mapping, meaning that each $k_i$ appears at least once in
$(k_{\si(1)},\ldots,k_{\si(p)})$. Here, we use the fact that $\ga$
respects $n_1\otimes\cdots\otimes n_r$ and contains no singleton.
\end{itemize}

Then, by applying Cauchy--Schwarz inequality over the set of indices
$(k_1,\ldots,k_q)$, we get
\[
\bigl\| C_\ga(f_1 \otimes\cdots\otimes f_r)
\bigr\|_{\ell^2}^2 \leq\|f_1\|_{\ell^2}^2
\bigl\| C_{\tilde{\ga}}(f_2 \otimes\cdots\otimes f_r)
\bigr\|_{\ell^2}^2,
\]
where we have used (possibly several times) the trivial property: for
any $g\dvt\{1,\ldots,N\}^2 \to\R$, $\sum_{k=1}^N g(k,k)^2 \leq
\sum_{k_1,k_2=1}^N g(k_1,k_2)^2$. We can now conclude by induction.
\end{pf}

Let us finally turn to the proof of
Proposition \ref{prophyp}, which is the hypercontractivity property
for homogeneous sums of free random variables.
We shall use
Lemma \ref{lemtechn} as a main ingredient.
The following elementary lemma will also be needed at some point.

%
\begin{lemma}\label{unif-bo}
For every integer $r\geq1$ and every sequence $X=(X_i)$ of random
variables, one has $|\varphi( X_{i_1} \cdots X_{i_{2r}} ) | \leq
\mu^X_{2^{r-1}}$, where $\mu^X_k:=\sup_{1\leq l\leq k , i\geq1}
\varphi
( X_i^{2l} )$.
\end{lemma}

\begin{pf}
For $r=1$, this corresponds to Cauchy--Schwarz inequality (see \cite
{nicaspeicher}). Assume that the result holds true up to $r-1$ ($r\geq
2$) for any sequence of random variables. By using Cauchy--Schwarz
inequality, we first get
%
%
\begin{eqnarray}
\label{cau-s}
&&\bigl|\varphi( X_{i_1} \cdots X_{i_{2r}} )\bigr|
\nonumber
\\
&&\quad=\bigl |\varphi \bigl( (X_{i_1} \cdots X_{i_{r}})
(X_{i_{r+1}} \cdots X_{i_{2r}}) \bigr)\bigr|
\\
&&\quad\leq \varphi \bigl( X_{i_1}^2 \cdots X_{i_{r-1}}X_{i_{r}}^2
X_{i_{r-1}} \cdots X_{i_2} \bigr)^{1/2}\varphi \bigl(
X_{i_{r+1}}^2 \cdots X_{i_{2r-1}}X_{i_{2r}}^2
X_{i_{2r-1}} \cdots X_{i_{r+2}} \bigr)^{1/2}.\nonumber
\end{eqnarray}
Denote by $X^2$ the sequence $X_1,X_1^2,X_2,X_2^2,\ldots.$ Then by
induction, we deduce from (\ref{cau-s}) that $|\varphi( X_{i_1}
\cdots X_{i_{2r}} )|\leq\mu^{X^2}_{2^{r-2}} \leq\mu^X_{2^{r-1}}$,
which concludes the proof.
\end{pf}

%
\begin{proposition}\label{prophyp}
Let $X_1,\ldots,X_N$ be centered freely independent random variables
and denote by $(\mu^N_k)$ the sequence of larger even moments, that is,
$\mu_k^N:=\sup_{1\leq i\leq N , 1\leq l\leq k} \varphi( X_i^{2l}
)$. Fix $d\geq1$, and consider a sequence of functions $f_N\dvt\{
1,\ldots,N\}^d\to\R$
satisfying the three basic assumptions \textup{(i)--(ii)--(iii)} of Theorem \ref
{MOOfree}. Define $Q_N$ through (\ref{qn}).
Then for every $r\geq1$, there exists a constant $C_{r,d}$ such that
%
%
\begin{equation}
\label{hyp-cont} \varphi \bigl( Q_N(X_1,
\ldots,X_N)^{2r} \bigr) \leq C_{r,d}
\mu^N_{2^{rd-1}} \Biggl( \sum_{j_1,\ldots,j_d=1}^N
f_N(j_1,\ldots,j_d)^2
\Biggr)^r.\vadjust{\goodbreak}
\end{equation}
\end{proposition}
\begin{pf}
The argument is in spirit quite close to ideas of \cite{kemp-speicher}.
Owing to Lemma \ref{condi}, it holds that
\begin{eqnarray*}
&&\varphi\bigl( Q_N(X_1,\ldots,X_N)^{2r}
\bigr)
\\
&&\quad= \mathop{\mathop{\sum_{1\leq j_1^1,\ldots,j^1_d \leq N}}_{
\vdots}
}_{ 1\leq j^{2r}_1,\ldots,j^{2r}_d \leq N} f_N\bigl(j^1_1,
\ldots,j^1_d\bigr) \cdots f_N
\bigl(j^{2r}_1,\ldots,j^{2r}_d\bigr)
\varphi\bigl( (X_{j_1^1} \cdots X_{j^1_d} )\cdots(X_{j^{2r}_1}
\cdots X_{j^{2r}_d}) \bigr)
\\
&&\quad= \sum_{(j^1_1,\ldots,j^{2r}_d) \in\ca_{2rd}^N} f_N\bigl(j^1_1,
\ldots ,j^1_d\bigr) \cdots f_N
\bigl(j^{2r}_1,\ldots,j^{2r}_d\bigr)
\varphi\bigl( (X_{j_1^1} \cdots X_{j^1_d} )\cdots(X_{j^{2r}_1}
\cdots X_{j^{2r}_d}) \bigr),
\end{eqnarray*}
where we have set, for every $R\geq1$,
\[
\ca_R^N:= \bigl\{(j_1,
\ldots,j_R) \in\{1,\ldots,N\}^R\dvt\mbox{for each }
i_1, \mbox{ there exists } i_2\neq i_1
\mbox{ such that } j_{i_1}=j_{i_2} \bigr\}.
\]
Bounding each term of the form $\varphi( (X_{j_1^1} \cdots
X_{j^1_d} )\cdots(X_{j^{2r}_1} \cdots X_{j^{2r}_d}) )$ of this sum
by means of Lemma \ref{unif-bo} leads to
\[
\varphi \bigl( Q_N(X_1,\ldots,X_N)^{2r}
\bigr) \leq\mu^N_{2^{rd-1}} \sum_{(j^1_1,\ldots,j^{2r}_d) \in\ca_{2rd}^N}
\bigl| f_N \bigl(j^1_1,\ldots,j^1_d
\bigr)\bigr| \cdots\bigl| f_N \bigl(j^{2r}_1,
\ldots,j^{2r}_d \bigr)\bigr|.
\]
Recall the notation $\mathcal{G}_\ast^c(d^{\otimes2r})$ and $C_\ga$
from the beginning of Section \ref{subsechyper}.
By taking into account that $f_N$ is assumed to vanish on diagonals, it
is easily seen that the above sum is equal to
\[
\sum_{(j^1_1,\ldots,j^{2r}_d) \in\ca_{2rd}^N} \bigl| f_N \bigl(j^1_1,
\ldots,j^1_d \bigr)\bigr| \cdots\bigl| f_N
\bigl(j^{2r}_1,\ldots,j^{2r}_d
\bigr)\bigr|=\sum_{\ga\in\mathcal
{G}_\ast^c(d^{\otimes2r})} C_\ga \bigl(
|f_N|^{\otimes2r} \bigr).
\]
Therefore, we may apply Lemma \ref{lemtechn} so as to deduce that
\[
\varphi \bigl( Q_N(X_1,\ldots,X_N)^{2r}
\bigr) \leq\mu^N_{2^{rd-1}}\cdot\bigl|\mathcal{G}_\ast^c
\bigl(d^{\otimes2r} \bigr)\bigr|\cdot\|f_N\|^{2r}_{\ell^2(\{
1,\ldots,N\}^d)},
\]
which is precisely (\ref{hyp-cont}) with $C_{r,d}=|\mathcal{G}_\ast^c(d^{\otimes2r})|$.
\end{pf}

\section*{Acknowledgements} We are grateful to Todd
Kemp for helpful comments and references about
hypercontractivity. Special thanks go to Roland Speicher,
who suggested a shorter proof for the hypercontractivity
property (Proposition \ref{prophyp}). Finally, we
thank an anonymous referee for a careful reading and
for his/her positive comments and constructive remarks.

This work has been supported in part by the two following (french) ANR
grants: `Exploration des Chemins Rugueux'
[ANR-09-BLAN-0114] and `Malliavin, Stein and Stochastic Equations with
Irregular Coefficients'
[ANR-10-BLAN-0121].

%


\printhistory

\end{document}